\documentclass[12pt,a4paper]{amsart}
\usepackage{url}
\usepackage{amsmath,amsthm,amsfonts,amssymb,latexsym}

\newtheorem{theorem}{Theorem}[section]

\newtheorem{lemma}[theorem]{Lemma}

\theoremstyle{definition}
\newtheorem{definition}[theorem]{Definition}

\newtheorem{question}[theorem]{Question}

\newcommand{\U}{\mathcal U}
\newcommand{\w}{\omega}

\newcommand{\IP}{\mathbb P}

\newcommand{\A}{\mathcal{A}}

\newcommand{\F}{\mathcal{F}}

\newcommand{\V}{\mathcal{V}}

\newcommand{\bigvid}{\hat{\ \ }}

\newcommand{\uhr}{\upharpoonright}

\newcommand{\name}[1]{\dot{#1}}
\newcommand{\la}{\langle}
\newcommand{\ra}{\rangle}

\newcommand{\Split}{\mathrm{Split}}

\newcommand{\forces}{\Vdash}

\newcommand{\hot}{\mathfrak}

\newcommand{\nothing}[1]{}

\title[$M$-separable spaces in the Miller model]{$M$-separable spaces of functions are productive in the Miller model}

\author{Du\v{s}an Repov\v{s}  and Lyubomyr Zdomskyy}

\address{Faculty of Education, and Faculty of Mathematics and Physics,
University of Ljubljana, \& Institute of Mathematics, Physics and Mechanics, 1000 Ljubljana, Slovenia.}
\email{dusan.repovs@guest.arnes.si}
\urladdr{http://www.fmf.uni-lj.si/\~{}repovs/index.htm}

\address{Institut f\"ur Diskrete Mathematik und Geometrie, Technische Universit\"at Wien, Wiedner Hauptstra\ss e 8-10/104, 1040 Wien, Austria.}
\email{lzdomsky@gmail.com}
\urladdr{http://dmg.tuwien.ac.at/zdomskyy/}


\thanks{The first author was partially supported by
the Slovenian Research Agency grants P1-0292 and N1-0083.
The second author would
like to thank  the Austrian Science Fund FWF (Grants I 2374-N35 and I 3709-N35)
 for generous support for this research.\\
 We are also grateful to the anonymous referee for careful reading of the manuscript.}

\begin{document}

\begin{abstract}
We prove that in the Miller model, every $M$-separable space of the form $C_p(X)$, where $X$ is metrizable and separable,
is productively $M$-separable, i.e., $C_p(X)\times Y$ is $M$-separable
for every countable $M$-separable $Y$.
\end{abstract}

\maketitle

\section{Introduction}

This paper is devoted to products of $M$-separable spaces.
A topological space $X$ is said  to be \emph{$M$-separable}, if for
every sequence $\la D_n:n\in\w\ra$ of dense subsets of $X$, one can
pick finite subsets $F_n\subset D_n$ so that $\bigcup_{n\in\w}F_n$ is dense, see \cite{BelBonMat09}.
This notion was introduced in
\cite{Sch99} where $M$-separable spaces of the form $C_p(X)$ were characterized.
Here $C_p(X)$ is the set of all continuous functions $f:X\to\mathbb R$ with the topology
inherited from the Tychonoff product $\mathbb R^X$.
It is obvious that second-countable spaces
(even spaces with a countable $\pi$-base) are $M$-separable.
Our main result is the following

\begin{theorem}\label{main}
In the Miller model, the product of any two $M$-separable spaces
is $M$-separable, provided that all dense subspaces of this product are separable and
one of the spaces is of the form $C_p(Z)$ for some Tychonoff space $Z$.

In particular, if $Y$ is a countable $M$-separable space and $X=C_p(Z)$ is $M$-separable for some
second-countable $Z$, then $X\times Y$ is $M$-separable.
\end{theorem}

By the Miller model we standardly mean a forcing extension of a model of GCH by adding a
generic filter for an iteration with countable supports of length $\w_2$ of the poset
introduced by  Miller in \cite{Mil84}. We give more details about this poset in the next section.
One of the key properties of this poset is the inequality $\hot u<\hot g$
proved in \cite{BlaLaf89, BlaShe89}, see \cite{Bla10} for more information on cardinal characteristics of the reals.
In particular, an equivalent form of this inequality established in \cite{Laf92}
will be crucial for our proof of Lemma~\ref{good_bound}.

Let us recall that a topological space
$X$ is said to have the \emph{Menger property} (or, alternatively, is a \emph{Menger space})
if for every sequence $\la \U_n : n\in\omega\ra$
of open covers of $X$ there exists a sequence $\la \V_n : n\in\omega \ra$ such that
each $\V_n$ is a finite subfamily of $\U_n$ and the collection $\{\cup \V_n:n\in\omega\}$
is a cover of $X$. This property was introduced by Hurewicz, and the current name
(the Menger property) is used because Hurewicz
proved in \cite{Hur25} that for metrizable spaces his property is equivalent to
a certain property of a base considered earlier  by Menger in \cite{Men24}.
The Menger property is central to the study of the $M$-separability of function spaces:
For a Tychonoff space $X$, $C_p(X)$ is $M$-separable
if and only if all finite powers of $X$ are Menger and $X$ admits a weaker separable metrizable
topology, see \cite[Theorem~2.9]{BelBonMatTka08} or \cite[Theorem~35]{Sch99}.
Let us also note that
by the main result of \cite{Zen80},
all finite powers of $C_p(X)$ are hereditarily separable
if all finite powers of $X$ are hereditarily Lindel\"of.
In particular, $C_p(Z)$ is hereditarily separable for second countable spaces $Z$.

Our paper is a further development of the ideas in \cite{RepZdo17,RepZdo18, Zdo??}.
However, the proof of Theorem~\ref{main} is
conceptually
different from those in these three papers,
since here we have to analyze the local structure of spaces of functions in the Miller model.
Also, unlike in \cite{Zdo??}, we were unable to achieve the optimal result
(which would be the consistency of the preservation of $M$-separability by finite products of
countable spaces), and affirmative answers to any of the last two items
in Question~\ref{qu1} would fill in this gap
by Lemma~\ref{covering_g_delta}.

The main result of \cite{Zdo??} states that in the Miller model, the product of any two second-countable spaces with the
Menger property
is Menger. Thus in this model the characterization mentioned above yields that for any two second-countable spaces
$Z_0,Z_1$, if $C_p(Z_0)$ and $C_p(Z_1)$ are $M$-separable, then so is
$C_p(Z_0)\times C_p(Z_1)$. Thus it is worth mentioning here that there are countable $M$-separable spaces
which cannot be embedded into $M$-separable spaces of the form $C_p(Z)$, and hence Theorem~\ref{main} indeed covers more
cases of $M$-separable spaces as the main result of \cite{Zdo??} combined with the characterization
in \cite{BelBonMatTka08,Sch99}. The easiest example of such a space is the \emph{Fr\'echet-Urysohn fan} $S_\w$, i.e., the factor space of the product $\w\times (\{0\}\cup\{1/n:n\in\w\})\subset\mathbb R^2$
obtained by identifying all points in $\w\times\{0\}$. It is obviously $M$-separable,
and it fails to have the countable fan tightness introduced
in \cite{Ark86}, whereas every $M$-separable space of the form $C_p(Z)$ has countable fan tightness by \cite[Corollary 2.10]{BelBonMatTka08}
and the latter property is hereditary.

On the other hand, there are many consistent examples under CH and weakenings thereof of
countable $M$-separable spaces with non-$M$-separable products, see, e.g.,
\cite{BarDow11, RepZdo10}.
As it was demonstrated in \cite[\S 6]{MilTsaZso16}, in all cases when such a
non-preservation result is known, one can obtain it by using
spaces of the form $C_p(Z)$, which is a yet another motivation behind Theorem~\ref{main}.

Theorem~\ref{main} seems to be  the best known approximation towards the answer
to the first item of following question which is central in this area. It was first asked in \cite{BelBonMatTka08} and then repeated
in several other papers. We refer the reader to
Definition~\ref{def1} for the  notions appearing in the last two items.
\begin{question}\label{qu1}
\begin{enumerate}
\item Is it consistent that the product of two countable $M$-separable spaces is $M$-separable?
Does this statement hold in the Miller model? Does it follow from $\hot u<\hot g$?
\item In the Miller model, does every countable $M$-separable space $X$ have a point $x$ (equivalently, densely many points $x$) such that $\zeta(X,x)\leq\w_1$?
\item In the Miller model, does every countable $M$-separable space $X$ have property $(\dagger)$?
\end{enumerate}
\end{question}

\section{Proof of Theorem~\ref{main}} \label{proofsection}

We divide the proof of Theorem~\ref{main} into a sequence of auxiliary statements.
More precisely, it will follow immediately from Lemmata
\ref{laver}, \ref{covering_g_delta}, \ref{loc_meng}, and \ref{cp_over_menger}.

\begin{definition} \label{def1}
\begin{enumerate}
\item
A family $\U$ of open subsets of a space $X$ is called
\emph{centered}, if $\cap\V\neq\emptyset$ for any $\V\in
[\U]^{<\w}$.
\item
A topological space $ \la X,\tau\ra$ is said to have property $(\dagger)$
if for every family $\mathsf R$ of size $\w_1$ of functions $R$
assigning to each countable centered family $\U$ of open subsets
of $X$ a sequence $R(\U)\in ([X]^{<\w}\setminus\{\emptyset\})^\w$ such that
$$ \forall U\in\U\:\forall^\infty n\in\w\:(R(\U)(n)\subset U),$$
there\footnote{Here $\forall^\infty$ means ``for all but finitely many''.} exists
$\mathsf U \in [[\tau\setminus\{\emptyset\}]^\w]^{\w_1}$ consisting
of countable centered families
such that
for all $O\in\tau\setminus\{\emptyset\}$ there exists
$\U\in\mathsf U$ such that for every $R\in\mathsf R$ there
exists $n\in\w$ with the property $R(\U)(n)\subset O$.
\item For a topological space $X$ and $x\in X$ we denote by $\zeta(X,x)$ the minimal cardinality
$\kappa$ such that for every sequence $\la A_n:n\in\w\ra$ such that $x\in\bar{A}_n$ for all $n$,
there exists a sequence $\la\la K^\alpha_n:n\in\w\ra:\alpha<\kappa\ra$ such that
$K^\alpha_n\in [A_n]^{<\w}$ for all $n,\alpha$, and for every open $U\ni x$ there exists $\alpha\in\kappa$
such that $U\cap K^\alpha_n\neq\emptyset$ for all $n\in\w$.
\item For a topological space $X$ we denote by $\zeta(X)$ the cardinal $\sup\{\zeta(X,x):x\in X\}$.
\end{enumerate}
\end{definition}

Spaces $X$ with $\zeta(X)\leq\w$ are exactly the spaces which are
\emph{weakly Fr\'echet in the strict sense} in the terminology of \cite{BelBonMat09, Sak07}.

In order to prove Theorem~\ref{main} we need to recall some details related to the Miller forcing.
By a \emph{Miller tree} we understand a subtree $T$ of $\w^{<\w}$ consisting of increasing finite sequences
such that the following conditions are satisfied:
\begin{itemize}
\item Every $t\in T$ has an extension $s\in T$ which  splits  in $T$, i.e.,
there are more than one immediate successors of $s$ in $T$;
\item If $s$ is splitting in $T$, then it has infinitely many successors in $T$.
\end{itemize}
The \emph{Miller forcing} is the collection $\mathbb M$ of all Miller trees ordered
by inclusion, i.e.,
smaller trees carry more information about the generic.
This poset was introduced in \cite{Mil84} and has since then found numerous applications
see, e.g., \cite{BlaShe89}.
We denote by $\IP_\alpha$ an iteration of length $\alpha$ of the Miller forcing
with countable support.
If $G$ is $\IP_\beta$-generic and $\alpha<\beta$, then we denote the intersection
$G\cap\IP_\alpha$ by $G_\alpha$.

For a Miller tree $T$ we shall denote by $\Split(T)$ the set of all splitting nodes
of $T$. For a node $t$ in a Miller tree $T$ we denote by $T_t$
the set $\{s\in T:s$ is compatible with $t\}$. It is clear that $T_t$ is also a Miller
tree. The \emph{stem} of a Miller tree $T$ is the shortest $t\in\Split(T)$. We denote the stem of $T$ by $T\la 0\ra$.
If $T_1\leq T_0 $ and $T_1\la 0\ra=T_0\la 0\ra$, then we write
$T_1\leq^0 T_0$.

The following lemma can be proved by an almost verbatim repetition of
the proof of \cite[Lemma~14]{Lav76},
see also \cite[\S 2]{Zdo??} for a more general form. Here by a real we mean a subset of $\w$.

\begin{lemma} \label{miller_like_laver}
Let $\name{x}$ be a $\IP_{\w_2}$-name for a real
and $p\in\IP_{\w_2}$. Then there exist $p'\leq p$ such that
$p'(0)\leq^0 p(0)$, and a finite set of reals $U_s$ for each $s\in\Split(p'(0))$,
such that for each $N\in\w$, $s\in \Split(p'(0))$,
and for all but finitely many immediate successors $t$ of $s$ in $p'(0)$ we have
$$ (p'(0))_t\hat{\ \ }p'\uhr [1,\w_2)\forces\exists u\in U_s\: (u\cap N=\name{x}\cap N).$$
\end{lemma}

A subset $C$ of $\w_2$ is called an
\emph{$\w_1$-club} if it is unbounded and for every $\alpha\in\w_2$ of cofinality $\w_1$,
if $C\cap\alpha$ is cofinal in $\alpha$ then $\alpha\in C$.

The following lemma will be the key part of the proof of Theorem~\ref{main}.

\begin{lemma} \label{laver}
In the Miller model every countable space $X$
such that $\{x\in X :\zeta(X,x)\leq\w_1\}$ is dense in $X$, has property $(\dagger)$.
\end{lemma}
\begin{proof}
We work in $V[G_{\w_2}]$, where $G_{\w_2}$ is $\IP_{\w_2}$-generic
and $\IP_{\w_2}$ is the iteration of length $\w_2$ with countable supports
of the Miller forcing.
Let us write $X$ in the form $\la\w,\tau\ra$
and let $\mathsf R=\{R_\alpha:\alpha<\w_1\}$ be such as in the definition of $(\dagger)$.
By a standard argument (see, e.g., the proof of \cite[Lemma~5.10]{BlaShe87})
there exists an $\w_1$-club $C\subset \w_2$ such that for every $\alpha\in C$
the following conditions hold:
\begin{itemize}
\item[$(i)$] $\tau\cap V[G_\alpha] \in V[G_\alpha]$ and
for every $x\in\w$ and every sequence $\la A_n:n\in\w\ra\in V[G_\alpha]$ of subsets of
$ \w$ containing $x$ in their closure,
there exists $\la\la K^\alpha_n:n\in\w\ra:\alpha<\w_1\ra\in V[G_\alpha]$
such as Definition~\ref{def1}(3);
\item[$(ii)$] $\{ R_\alpha (\U):\alpha\in\w_1,\U\in [\tau\cap V[G_\alpha]]^\w\cap V[G_\alpha] $ is centered$\} \in V[G_\alpha] $;
\item[$(iii)$] For every $A\in\mathcal P(\w)\cap V[G_\alpha]$ the interior
$\mathit{Int}(A)$ also belongs to $V[G_\alpha]$.
\end{itemize}
Standardly,
there is no loss of generality
in assuming that $0\in C$. We claim that
$$\mathsf U:=\{\U\in [\tau\setminus\{\emptyset\}]^\w\cap V:\U \mbox{ is centered}\}$$
is a witness for $\la\w,\tau\ra$ satisfying $(\dagger)$. Suppose,
contrary to our claim, that there exists
$A\in\tau\setminus\{\emptyset\}$ such that for every $\U\in \mathsf
U$ there exists $\alpha\in\w_1$ such that
$R_\alpha(\U)(n)\not\subset A$ for all $n\in\w$. Let $\name{A}$ be
a $\IP_{\w_2}$-name for $A$ and $p\in\IP_{\w_2}$
a condition forcing the above statement. Without loss of
generality,
we may assume that 
 there exists $N\in\w$ such that
$\zeta (X,N)\leq \w_1$ and $p\forces N\in \name{A}$.

Applying Lemma~\ref{miller_like_laver}
to $\name{x}:=\name{A}$, 
we get a condition $p'\leq p$ such that $p'(0)\leq^0 p(0)$, and a finite set
$\U_s\subset\mathcal P(\w)$
for every $s\in \Split(p'(0))$, such that for each $n\in\w$,
$s\in \Split(p'(0))$, and for all but finitely many
immediate successors $t$ of $s$ in $p'(0)$ we have
$$ p'(0)_t\bigvid p'\uhr[1,\w_2)\forces \exists U\in \U_s\: (\name{A}\cap n= U\cap n). $$
Of course, any $p''\leq p'$ also has the above property  with the same
$\U_s$'s. However, the stronger $p''$ is, the more elements
of $\U_s$ might play no role any more. Therefore throughout the rest of the proof we
shall call $U\in\U_s$ \emph{void for $p''\leq p'$ and $s\in \Split(p''(0))$},
if there exists $n\in\w$ such that for all but finitely many
immediate successors $t$ of $s$ in $p''(0)$ there is \emph{no}
$q\leq p''(0)_t\bigvid p''\uhr[1,\w_2)$
with the property $q\forces \name{A}\cap n= U\cap n.$
Note that for any $p''\leq p'$ and $s\in \Split(p''(0))$ there exists $U\in\U_s$ which is non-void for $p'',s$.
Two cases are possible.
\smallskip

Case $a)$ \ For every $p''\leq p'$ there exists $s\in \Split(p''(0))$ and
a non-void $U\in\U_s$ for $p'',s$ such that
$N\in\mathit{Int}(U)$. Let
$\U$ be the collection of $\mathit{Int}(U)$ for all $U$ as above.
It follows from the
above that $p'$ forces that there exists $\alpha\in\w_1$ such that
$R_\alpha(\U)(n)\not\subset \name{A}$ for
all $n\in\w$. Passing to a stronger condition if
necessary, we may additionally assume that $p'$ decides $\alpha$.

Fix a non-void $U$ for $p', s$, where $s\in \Split(p'(0))$,
such that $N\in \mathit{Int}(U)$ (and hence $\mathit{Int}(U)\in\U$).
It follows from the above that there exists $m$ such that
$R_\alpha(\U)(k)\subset \mathit{Int}(U) $ for all $k\geq m$. Let $n\in\w$ be such that
$R_\alpha(\U)(m)\subset n$. By the definition of being non-void, 
there are infinitely many
immediate successors $t$ of $s$ in $p'(0)$ for which there exists
$q_t\leq p'(0)_t\bigvid p'\uhr[1,\w_2)$
with the property $q_t\forces \name{A}\cap n= U\cap n.$ Then for any $q_t$
as above we have that $q_t$ forces $R_\alpha(\U)(m)\subset \name{A}$
because $R_\alpha(\U)(m)\subset U\cap n$, which contradicts the fact that
$q_t\leq p'$ and $p'\forces R_\alpha(\U)(m)\not\subset \name{A} $.
\smallskip

Case $b)$ \ There exists $p''\leq p'$ such that for all $s\in \Split(p''(0))$,
every $U\in\U_s$ with $N\in\mathit{Int}(U)$
is void for $p'',s$. Note that this implies that every $U\in\U_s$ with $N\in \mathit{Int}(U)$, $U$
is void for $q,s$ for all $q\leq p''$ and $s\in \Split(q(0))$.
Let $\la D_k:k\in\w\ra\in V$ be a sequence
of subsets of $\w$ such that
$$ \big\{D_k:k\in\w\big\} = \big\{\w\setminus U: U\in\bigcup_{s\in \Split(p''(0))}\U_s, N\not\in\mathit{Int}(U)\big\}.$$
Item $(i)$ above yields
a sequence
$\la\la K^\alpha_k:k\in\w\ra:\alpha<\w_1\ra\in V$ such that $K^\alpha_k\in [D_k]^{<\w}$ for all $k$, and for every neighborhood
$O\in\tau$ of $N$ there exists $\alpha\in\w_1$
such that $K^\alpha_k\cap O\neq\emptyset$ for all $k\in\w$.
Let $p^{(3)}\leq p''$ decide $\alpha$ which has the property stated above for $\name{A}$.
Fix $U\in\U_{p^{(3)}(0)\la 0\ra}$ non-void for $p^{(3)}, p^{(3)}(0)\la 0\ra$.
Then $N\not\in\mathit{Int}(U)$ by the choice of $p''$ and hence there exists
$k$ such that $\w\setminus U=D_k$. It follows that
$K^\alpha_k\cap U=\emptyset$ because $K^\alpha_k\subset D_k$. On the other hand,
since $U$ is non-void for $p^{(3)}, p^{(3)}(0)\la 0\ra$, for $n=\max K^\alpha_k+1$
we can find infinitely many
immediate successors $t$ of $p^{(3)}(0)\la 0\ra$ in $p^{(3)}(0)$ for which there exists
$q_t\leq p^{(3)}(0)_t\bigvid p^{(3)}\uhr[1,\w_2)$
forcing $\name{A}\cap n= U\cap n.$ Then any such $q_t$ forces $K^\alpha_k\cap\name{A}=\emptyset$
(because $K^\alpha_k\subset n$ and $K^\alpha_k\cap U=\emptyset$),
contradicting the fact that $p^{(3)}\geq q_t$ and $p^{(3)}\forces K^\alpha_k\cap\name{A}\neq\emptyset$
for all $k$.
\smallskip

Contradictions obtained in cases $a)$ and $b)$ above imply that
$\mathsf U$
is a witness for $\la\w,\tau\ra$ having $(\dagger)$, which completes the  proof of Lemma~\ref{laver}.
\end{proof}

It is well-known \cite{BlaShe89} that in the Miller model there exists an ultrafilter $\F$
generated by $\w_1$-many sets, say $\{F_\alpha:\alpha\in\w_1\}$. It plays an important role in the proof
of the following

\begin{lemma} \label{good_bound}
In the Miller model, for every $M$-separable space $X$
and every decreasing sequence $\la D_n:n\in\w\ra$ of countable dense subsets of $X$, there exists a sequence
$\la\la K^\alpha_n:n\in\w \ra:\alpha\in\w_1\ra$ such that
\begin{enumerate}
\item $K^\alpha_n\in [D_n]^{<\w}$ for all $n\in\w$ and $\alpha\in\w_1$; and
\item for every open non-empty $O\subset X$, there exists $\alpha\in\w_1$ such that $O\cap K^\alpha_n\neq\emptyset$
for all $n\in\w$.
\end{enumerate}
\end{lemma}
\begin{proof}
Let us write $D_n$ in the form $\{d^n_k:k\in\w\}$ and fix an increasing function $f\in\w^\w$ such that
for every open non-empty $O\subset X$ there are infinitely many
$n\in\w$ such that $O\cap\{d^n_k:k\leq f(n)\}\neq\emptyset$.
(This is possible due to the $M$-separability of $X$.)
Let us denote by $U_O$ the set of all such $n$.
By \cite[Theorem~10]{Laf92} combined with \cite[Theorems~1,2]{BlaLaf89}\footnote{As noted
by the referee, these results from \cite{BlaLaf89, Laf92} only give certain finite-to-one function.  However,
it is rather standard
and not too difficult to derive from this function  an increasing sequence $\la m_i:i\in\w\ra\in\w^\w$ with the properties we need in this proof.},
for the family $\U=\{U_O:O$ is an open non-empty subset of $X\}$
there exists an increasing sequence $\la m_i:i\in\w\ra\in\w^\w$ such that
one of the following
options takes place:
\begin{itemize}
\item For every $O$, the set $\bigcup\{[m_i,m_{i+1}):U_O\cap [m_i,m_{i+1})\neq\emptyset\}$ belongs to $\F$; or
\item For every $A\in [\w]^\w$, there exists
$O$ such that $U_O\subset^* \bigcup\{[m_i,m_{i+1}):i\in A\}$.
\end{itemize}
Suppose that the second option takes place and let $\A\subset[\w]^\w$ be an infinite (and hence uncountable)
maximal almost disjoint family. For every $A\in\A$, fix an open non-empty subset $O(A)$
of $X$ such that $U_{O(A)}\subset^* \bigcup\{[m_i,m_{i+1}):i\in A\}$ and note that
this implies $|U_{O(A)}\cap U_{O(A')}|<\w$ for any distinct $A,A'\in\A$.
On the other hand, since $X$ is separable and $\A$ is uncountable, there are distinct $A,A'\in\A$ such that $O(A)\cap O(A')\neq\emptyset$,
and hence $U_{O(A)\cap O(A')}$ is infinite, contradicting the fact that
$U_{O(A)\cap O(A')}\subset U_{O(A)}\cap U_{O(A')}$ and the latter intersection is finite.

Thus the first option must take place. For every $\alpha\in\w_1$ and $n\in\w$ let $i_{\alpha,n}$ be the minimal number $i$
such that $m_i\geq n$ and $[m_i,m_{i+1})\cap F_\alpha\neq\emptyset$.
We claim that
the sequences
$$\big\la K^\alpha_n=\{d^l_k:l\in [m_{i_{\alpha,n}}, m_{i_{\alpha,n}+1}),k\leq f(l)\}\ :\ n\in\w\big\ra$$
are as required. Indeed, given $O$, find $\alpha$
such that
$F_\alpha\subset \bigcup\{[m_i,m_{i+1}):U_O\cap [m_i,m_{i+1})\neq\emptyset\}$. Now for every $n\in\w$ we have
$$ K^\alpha_n\cap O= \{d^l_k:l\in [m_{i_{\alpha,n}}, m_{i_{\alpha,n}+1}),k\leq f(l)\}\cap O,$$
and the latter intersection is non-empty because $[m_{i_{\alpha,n}}, m_{i_{\alpha,n}+1})\cap F_\alpha\neq\emptyset $,
hence also $[m_{i_{\alpha,n}}, m_{i_{\alpha,n}+1})\cap U_O\neq\emptyset $,
and thus for every $l\in [m_{i_{\alpha,n}}, m_{i_{\alpha,n}+1})\cap U_O$ we have
$O\cap \{d^l_k:k\leq f(l)\}\neq\emptyset$. This completes the proof of Lemma~\ref{good_bound}.
\end{proof}

There is a
natural linear preorder $\leq_\F$ on $\w^\w$ associated to $\F$ defined as follows:
$x\leq_\F y$ if and only if $\{n\in\w:x(n)\leq y(n)\}\in \F$.
By \cite[Theorem~3.1]{BlaMil99}, in the Miller model, for every $X\subset\w^\w$
of size $\w_1$ there exists $b\in\w^\w$ such that $x\leq_\F b$ for all $x\in X$.
As an easy consequence thereof we get the following fact:
Suppose that $\la D_n:n\in\w\ra$ is a sequence of countable sets and
$A_{\alpha,n}\in [D_n]^{<\w}$ for all $\alpha\in\w_1$ and $n\in\w$. Then there exists a sequence
$\la A_n:n\in\w\ra$ such that $A_n\in [D_n]^{<\w}$ for all $n$, and
$\{n:A_{\alpha,n}\subset A_n\}\in \F$ for all $\alpha\in\w_1$.

\begin{lemma} \label{covering_g_delta}
In the Miller model, suppose that $|X|=|Y|=\w$, $X$ satisfies $(\dagger)$, and $Y$ is $M$-separable.
Then $X\times Y$ is $M$-separable.
\end{lemma}
\begin{proof}
Let $\la D_n:n\in\w\ra$ be a sequence of dense subsets of $X\times Y$. By \cite[Lemma~2.1]{GruSak11},
 there is no loss of generality in
assuming that $D_{n+1}\subset D_n$ for all $n$. Given an open non-empty subset $U$ of
$X$, for every $n\in \omega$ set $D^{U}_n=\{y\in Y:\exists x\in U (\la x,y\ra\in D_n)\}$
and note that $D^U_n$ is dense in $Y$. Given a countable centered family $\U$ of open subsets of $X$,
fix a decreasing sequence $\la U_{\U,n}:n\in\w\ra$ of open subsets of $X$ such that
for every $U\in\U$,  there exists $n\in\w$ such that $U_{\U,n}\subset U$.
By Lemma~\ref{good_bound} there exists a sequence
$$ \big\la \la L^{\alpha,\U}_n:n\in\w \ra \: :\: \alpha\in\w_1\ra $$
such that $L^{\alpha,\U}_n\in [D^{U_{\U,n}}_n]^{<\w}$ for all $n,\alpha$, and for every open non-empty
$V\subset Y$, there exists $\alpha$ such that $L^{\alpha,\U}_n\cap V\neq\emptyset$ for all $n$.
Let us find $K^{\alpha,\U}_n\in [U_{\U,n}]^{<\w}$ such that
for every $y\in L^{\alpha,\U}_n$, there exists $x\in K^{\alpha,\U}_n$
such that $\la x,y\ra\in D_n$. For every $\alpha,\beta\in\w_1$ and $n\in\w$, set
$R_{\alpha,\beta}(\U)(n)= K^{\alpha,\U}_{\min( F_\beta\setminus n)}$.
%\footnote{Formally,
%$R_{\alpha,\beta}(\U)$ should have domain $\w$. We can define
%it arbitrarily on $\w\setminus F_\beta$ in such a way that
%for all $U\in\U$ and all but finitely many $n\in\w\setminus F_\beta$ we have $R_{\alpha,\beta}(\U)(n)\subset U$,
%the restriction $R_{\alpha,\beta}(\U)\uhr (\w\setminus F_\beta)$ will play no role
%anyway.}.
Note that $\mathsf R=\{R_{\alpha,\beta}:\alpha,\beta\in\w_1\}$ is
such as in the definition of $(\dagger)$ because
$K^{\alpha,\U}_n\subset U$ for all $U\in\U$ and all but finitely
many $n\in\w$. It follows that there exists a family
$\mathsf U$ of countable centered families $\U$ of open subsets of
$X$ of size $|\mathsf U|=\w_1$, and such that for every open
non-empty $O\subset X$, there exists $\U\in\mathsf U$ such that for
all $\alpha,\beta\in\w_1$, there exists
$n\in F_\beta$ with the property $K^{\alpha, \U}_n\subset O$.
Since $\F$ is an ultrafilter, it follows that for all
$\alpha\in\w_1$, there exists $\xi\in\w_1$
with the property $K^{\alpha, \U}_n\subset O$ for all $n\in F_\xi$.

Since $|\U|=\w_1$, there exists a sequence $\la M_n:n\in\w\ra$ such
that $M_n\in [D_n]^{<\w}$ and for every $\U\in\mathsf U$ and
$\alpha,\beta\in\w_1$, we have
$$\big\{n\in\w : M_n\supset (K^{\alpha,\U}_n\times L^{\alpha,\U}_n)\cap D_n\big\}\in\F. $$
We claim that $\bigcup_{n\in\w}M_n$ is dense in $X\times Y$. Indeed,
let us fix an open non-empty subset of $X\times Y$ of the form
$O\times V$ and find $\U\in\mathsf U$ as above. Let $\alpha$ be such
that $L^{\alpha,\U}_n\cap V\neq\emptyset$ for all $n\in\w$. Pick
$\beta\in\w_1$ such that
$$F_\beta\subset \big\{n\in\w : M_n\supset (K^{\alpha,\U}_n\times L^{\alpha,\U}_n)\cap D_n\big\}. $$
Let $\xi\in\w_1$ be such that
$K^{\alpha, \U}_n\subset O$ for all $n\in F_\xi$.
Then for every $n\in F_\beta\cap F_\xi$, we have
$$
\emptyset\neq (O\times V)\cap (K^{\alpha,\U}_n\times
L^{\alpha,\U}_n)\subset (O\times V)\cap M_n,
$$
which completes the proof of Lemma~\ref{covering_g_delta}.
\end{proof}

Next lemma gives consistent examples of countable spaces $X$ such
that $\zeta(X)\leq\w_1$.

\begin{lemma}\label{loc_meng}
In the Miller model, suppose that $X=\la \w,\tau\ra$ is a
topological space and $x\in\w$ is such that $\U=\{U\in\mathcal
P(\w):x\in\mathit{Int}(U)\}$ is Menger. Then $\zeta(X,x)\leq \w_1$.
\end{lemma}
\begin{proof}
For every $n\in\w$, fix $A_n=\{a^n_k:k\in\w\}\subset\w$ such that
$x\in\bar{A}_n$. For every $U\in\U$, set
$$\phi(U)(n)=\min\{k:a^n_k\in U\} \mbox{ \ and \ } \Phi(U)=\{z\in\w^\w: \forall n\:(z(n)\leq \phi(\U)(n) )\}$$
and note that $\Phi$ is a compact-valued map from $\U$ to $\w^\w$.
We claim that it is upper semicontinuous, i.e., for every open
$W\subset\w^\w$ containing $\Phi(U)$ for some $U\in\U$, there exists
an open neighborhood $O$ of $U$ in $\U$ such that $\Phi(U')\subset
W$ for all $U'\in\U\cap O$.
For $U,W$ as above find $m\in\w$ such that
$$ \Phi(U)=\prod_{n\in\w}(\phi(U)(n)+1)\subset \prod_{n\leq m}(\phi(U)(n)+1)\times\prod_{n>m}\w\subset W.$$
Set
$O=\{U'\in\U:\forall n\leq m (\phi(U')(n)\leq \phi(U)(n))\}$ and note that
$O$ is open in $\mathcal P(\w)$ and $\Phi(U')\subset W$ for all $U'\in O\cap \U$.

Since $\U$ is Menger and $\Phi$ is compact-valued and upper semicontinuous,
$Z:=\bigcup_{U\in\U}\Phi(U)\subset\w^\w$ is Menger by \cite[Lemma~1]{Zdo05}.
Applying \cite[Lemma~2.3]{Zdo??}, we conclude that there exists $Y\in [\w^\w]^{\w_1}$
such that for every $z\in Z$ (in particular, for every $z$ of the form $\phi(U)$, where
$U\in\U$) there exists $y\in Y$ such that $z(n)\leq y(n)$ for all $n\in\w$.
It follows from the above that $K^y_n=\{a^n_k:k\leq y(n)\}$, where $y\in Y$ and
$n\in\w$, are witnessing for $\zeta(X,x)\leq\w_1$.
\end{proof}

\begin{lemma} \label{cp_over_menger}
Suppose that $X$ is a Tychonoff space such that $X^n$ is Menger for all $n\in\w$, and $0\in A\in[C_p(X)]^\w$
is such that $0$ is a limit point of $A.$ Then
$\U=\{U\in\mathcal
P(A):0\in\mathit{Int}(U)\}$ is Menger as a subspace of $\mathcal P(A)$,
where the interior is considered in the topology on $A$ inherited from $C_p(X)$.
\end{lemma}
\begin{proof}
By the definition of the topology of $C_p(X)$ we have that
$$\U=\bigcup_{n,m\in\w}\bigcup_{\vec{x}=\la x_0,\ldots, x_{n-1}\ra\in X^n}\uparrow U_{n,m,\vec{x}} , $$
where $U_{n,m\vec{x}}=\{a\in A: \forall i<n\:(a(x_i)<1/m)\}$
and $\uparrow B=\{B'\subset A:B\subset B'\}$ for all $B\subset A$. In the same way as in Lemma~\ref{loc_meng}, 
we can check that the map
$$X^n\ni\vec{x}\mapsto\uparrow U_{n,m\vec{x}}\subset\mathcal P(A)$$
is compact-valued and upper semicontinuous for all $n,m\in\w$, and hence by \cite[Lemma~1]{Zdo05} $\U$ is Menger
being a countable union of its Menger subspaces.
\end{proof}

Finally, we have all necessary ingredients for the
proof of Theorem~\ref{main}. It suffices to prove that in the Miller model the product of any two countable $M$-separable spaces
$X, Y$ is $M$-separable, provided that $X$ is a subspace of $C_p(Z)$ and $C_p(Z)$ is $M$-separable.
By \cite[Theorem~35]{Sch99}, we have that $Z^n$ has the Menger property for all $n\in\w$, and
hence for every $x\in X$, the family $\U=\{U\in\mathcal
P(X):x\in\mathit{Int}(U)\}$ is Menger as a subspace of $\mathcal P(X)$ by Lemma~\ref{cp_over_menger}.
Applying Lemma~\ref{loc_meng}, we conclude that
$\zeta(X)\leq \w_1$, and hence $X$ has property $(\dagger)$ by Lemma~\ref{laver}.
It remains to apply Lemma~\ref{covering_g_delta}. \hfill $\Box$

\end{document}